\documentclass[12pt]{amsart}
\usepackage{etex} 

\usepackage[letterpaper, margin=1in]{geometry} 

\usepackage[colorlinks,linkcolor=blue,citecolor=blue]{hyperref}
\usepackage{amsmath}
\usepackage{amssymb}
\usepackage{amsthm}
\usepackage{amsfonts}
\usepackage{booktabs}
\usepackage[usenames,dvipsnames,svgnames]{xcolor}
\usepackage{epsfig}
\usepackage{fancyvrb} 
\usepackage{hyperref}
\usepackage{mathrsfs}
\usepackage{mathtools}
\usepackage{pictexwd, dcpic}
\usepackage{stmaryrd}
\usepackage{soul} 
\newtheorem{theorem}{Theorem}[section]
\newtheorem{lemma}[theorem]{Lemma}
\newtheorem{proposition}[theorem]{Proposition}
\newtheorem{definition}[theorem]{Definition}

\newtheorem{corollary}[theorem]{Corollary}

\theoremstyle{definition}
\newtheorem{remark}[theorem]{Remark}
\newtheorem{example}[theorem]{Example}
\newtheorem{question}[theorem]{Question}
\newtheorem{construction}[theorem]{Construction}

\def\s{\sigma}

\def\Aut{\mathrm{Aut}}

\def\Z{\mathbb{Z}}

\title{The space of circular orderings and semiconjugacy}
\date{\today}

\begin{document}

\author[Idrissa Ba]{Idrissa Ba}
\address{Department of Mathematics\\
University of Manitoba \\
Winnipeg \\
MB Canada R3T 2N2} \email{ba162006@yahoo.fr}
\urladdr{ https://server.math.umanitoba.ca/~idrissa/}
\thanks{Idrissa Ba was partially supported by a Postdoctoral Fellowship at the University of Manitoba}

\author[Adam Clay]{Adam Clay}
\address{Department of Mathematics\\
University of Manitoba \\
Winnipeg \\
MB Canada R3T 2N2} \email{Adam.Clay@umanitoba.ca}
\urladdr{http://server.math.umanitoba.ca/~claya/} 
\thanks{Adam Clay was partially supported by NSERC grant RGPIN-2020-05343}
\maketitle

\begin{center}
\today
\end{center} 

\begin{abstract}
Work of Linnell shows that the space of left-orderings of a group is either finite or uncountable, and in the case that the space is finite, the isomorphism type of the group is known---it is what is known as a Tararin group.  By defining semiconjugacy of circular orderings in a general setting (that is, for arbitrary circular orderings of groups that may not act on $S^1$), we can view the subspace of left-orderings of any group as a single semiconjugacy class of circular orderings. Taking this perspective, we generalize the result of Linnell, to show that every semiconjugacy class of circular orderings is either finite or uncountable, and when a semiconjugacy class is finite, the group has a prescribed structure.  We also investigate the space of left-orderings as a subspace of the space of circular orderings, addressing a question of Baik and Samperton.
\end{abstract}

\maketitle

\section{Introduction}

A \textit{left-ordering} of a group $G$ is a strict, total ordering $<$ of $G$ such that $g<h$ implies $fg<fh$ for all $f, g, h \in G$.  If a group admits such an ordering, it is called \textit{left-orderable}, a pair $(G, <)$ is called a \textit{left-ordered group}.  
A {\it left-circular ordering} of a group $G$ is a map $c: G^3 \rightarrow \{ 0, \pm 1\}$ which is an orientation cocycle on triples of elements of $G$, see Section \ref{background} for details. We will often simply say ``circular ordering" for short.

For a fixed group $G$, the set of left-orderings of a group and the set of circular orderings of a group can both be topologized to produce compact topological spaces denoted $\mathrm{LO}(G)$ and $\mathrm{CO}(G)$ respectively.  As every left-ordering $<$ of a group corresponds uniquely to a certain ``degenerate" circular ordering, there is a natural embedding $\iota: \mathrm{LO}(G) \rightarrow \mathrm{CO}(G)$.

The groups for which $\mathrm{LO}(G)$ is finite are completely classified, they are known as the \textit{Tararin groups} $T_k$ and are described in Section \ref{background}.   Each group $T_k$ admits precisely $2^k$ left-orderings, with the structure of the orderings being well understood; they all arise lexicographically from a certain rational series $T_k \triangleright T_{k-1} \triangleright \dots \triangleright \{ id\}$ \cite{Ta}.  This result was the beginning of an investigation into other possible cardinalities of $\mathrm{LO}(G)$, with Zenkov later proving that $\mathrm{LO}(G)$ is either finite or uncountable whenever $G$ is locally indicable \cite{Ze}.  Linnell then improved this result via a clever application of compactness of $\mathrm{LO}(G)$, showing that no matter the group, $\mathrm{LO}(G)$ is either finite or uncountable \cite{Li}.

Following these results, an investigation of $\mathrm{CO}(G)$ showed that it behaves in a similar way:  It is either finite or uncountable no matter the group $G$ \cite{CMR}.  Moreover, when it is finite, $G$ must be a semidirect product $T_k \ltimes \mathbb{Z}/n \mathbb{Z}$ of a Tararin group and a finite cyclic group, with the cyclic group acting on the rational series quotients of $T_k$ in a specific way.

In this paper, we take the perspective that $\iota(\mathrm{LO}(G)) \subset \mathrm{CO}(G)$ is but a single semiconjugacy class of circular orderings, and view the result of Linnell as an analysis of the cardinality of this particular semiconjugacy class.  Tararin's result then gives a structure theorem, describing precisely the groups for which this particular semiconjugacy class in finite.  Our work generalizes the works of Linnell and Tararin, by showing that every semiconjugacy class of circular orderings of a group $G$ is either finite our uncountable, and by giving a structure theorem for the groups which admit finite semiconjugacy classes of circular orderings.  We prove:

\begin{theorem} \label{Thm2}
	\label{size of sets}
	Let $G$ be a circularly orderable group, and let $S \subset \mathrm{CO}(G)$ be a semiconjugacy class of circular orderings. Then either $S$ is uncountable, or $|S|=2^k$ and one of the following holds:
	\begin{enumerate}
	\item No circular ordering in $S$ is isolated in $\mathrm{CO}(G)$, and $G$ fits into a short exact sequence
	\[ 1 \rightarrow T_k \rightarrow G \rightarrow C \rightarrow 1
	\] 
	where $C$ is a nontrivial subgroup of $S^1$, $T_k$ is a Tararin group of rank $k$, and the map $G \rightarrow C$ is the rotation number map $\mathrm{rot}_c : G \rightarrow \mathbb{R} /\mathbb{Z}$ of one (and hence every) $c \in S$.  Moreover if $C$ is finite, then the action of $C$ on $T_k/T_{k-1}$ is trivial.
	\item All circular orderings in $S$ are isolated, $\mathrm{CO}(G)$ is finite, and thus $G \cong T_k \ltimes \mathbb{Z}/n\mathbb{Z}$ where $n$ is even and the action of $\mathbb{Z}/n\mathbb{Z}$ on $T_k/T_{k-1}$ is by inversion.
	\end{enumerate}
\end{theorem}

Note that we place no conditions on the group $G$ (such as countability, or admitting an action on $\mathbb{R}$ or $S^1$).  To avoid imposing these additional hypotheses, we have to first generalize the usual definitions of semiconjugacy and rotation number appearing in the literature so that they are applicable to arbitrary circularly orderable groups (See \cite{Gh}, \cite{Man2} and \cite{KKM} for reviews of classical results). This necessity arises from the fact that the typical definitions apply only to actions on $\mathbb{R}$ or $S^1$, yet there are circularly orderable groups (and left-orderable groups) which do not act on either of these $1$-manifolds.  In fact, there are groups admitting left-orderings (resp. circular orderings) that do not arise from actions on $\mathbb{R}$ (resp. actions on $S^1$), despite the fact that these groups embed in $\mathrm{Homeo}_+(\mathbb{R})$ (resp. $\mathrm{Homeo}_+(S^1)$), see Section \ref{not natural} .  In particular, the situation is worse than not admitting a dynamical realization---the orders we construct cannot come from from \textit{any} action by homeomorphisms  on $\mathbb{R}$ or $S^1$ by following the ``standard recipe" (see Construction \ref{standard_construction}).

The partition of $\mathrm{CO}(G)$ into semiconjugacy classes, one of them being $\iota(\mathrm{LO}(G))$, also provides a new perspective in tackling a question of Baik and Samperton \cite[Question 2.9]{BS}.  In their work, they investigated the properties of $\iota(\mathrm{LO}(G))$ by defining the collection of \textit{genuine} circular orderings to be $\mathrm{CO}_g(G) = \mathrm{CO}(G) \setminus \iota(\mathrm{LO}(G))$.     They show that $\overline{\mathrm{CO}_g(\mathbb{Z}^n)} = \mathrm{CO}(\mathbb{Z}^n)$, and ask whether or not one can determine the closure $\overline{\mathrm{CO}_g(G)}$ of genuine circular orderings for a given group $G$.  We improve their result, and show:

\begin{theorem}
\label{main2}
For the following groups, $\mathrm{CO}_g(G)$ is dense in $\mathrm{CO}(G)$:
\begin{enumerate}
\item Finitely generated nilpotent groups.
\item Braid groups.
\item Fundamental groups of closed Seifert fibred manifolds with base orbifold $S^2(\alpha_1, \dots, \alpha_n)$.
\end{enumerate}
\end{theorem}

These results follow from a general method for approximating certain kinds of left-orderings by genuine circular orderings, see Proposition \ref{genuine closure}.

\subsection{Organization}
The paper is organized as follows.  In Section \ref{background} we review notation and background on left-orderable and circularly orderable groups.  In Section \ref{not natural} we prove that there exist groups admitting left-orderings (resp. circular orderings) that do not arise from any action by homeomorphisms on $\mathbb{R}$ (resp. $\mathbb{S}^1$), supporting the necessity of a generalized notion of semiconjugacy.  In Section \ref{semiconjugacy_section} we develop semiconjugacy of circular orderings, as well as rotation and translation numbers in a general setting.  In Section \ref{main_proof} we prove Theorem  \ref{Thm2}.  Section \ref{genuine} addresses the problem of Baik and Samperton, and the limit set of genuine circular orderings, proving Theorem \ref{main2}.

\subsection{Acknowledgements} The authors would like to thank Crist\'{o}bal Rivas and Ty Ghaswala for their helpful discussions.

\section{Background and notation}
\label{background}

Every left-ordering $<$ of a group $G$ has a corresponding positive cone $P= \{ g \in G \mid g>id \}$, satisfying (i) $P \cdot P \subset P$ and (ii) $P \cup P^{-1} = G \setminus \{ id \}$.  Conversely, and subset $P$ of $G$ satisfying (i) and (ii) defines a left-ordering $<$ of $G$ according to the rule $g<h$ if and only if $g^{-1}h \in P$.

The set of all positive cones is denoted $\mathrm{LO}(G)$, and is considered as a subspace of the power set $\{0,1\}^G$, endowed with the product topology.  As such, a subbasis for the topology on $\mathrm{LO}(G)$ is the collection of sets 
\[ U_g = \{ P \in \mathrm{LO}(G) \mid g \in P\}
\]
where $g \in G \setminus \{id\}$.
It happens that $\mathrm{LO}(G)$ is closed in $\{0,1\}^G$, hence compact.  We will often implicitly identify each positive cone $P \in \mathrm{LO}(G)$ with its corresponding ordering $<$, and speak of $\mathrm{LO}(G)$ as being the space of all left-orderings of $G$.

Suppose $\phi : G \rightarrow H$ is an injective homomorphism between groups, and  $H$ is equipped with a left-ordering $<$.  Then we can define a left-ordering $<^{\phi}$ of $G$ according to the rule $g<^{\phi}h$ if and only if $\phi(g)< \phi(h)$ for all $g, h \in G$.  When $H=G$ and $\phi$ is an inner automorphism, say there exists $f \in G$ such that $\phi(g) = fgf^{-1} $ for all $g \in G$, then we denote $<^{\phi}$ by $<^f$.  Note that this gives a right $G$-action on the set of left-orderings $\mathrm{LO}(G)$, defined by $< \cdot f = <^f$ for all $f \in G$.   Correspondingly, if $P$ is the positive cone of a left-ordering $<$ of $G$, then $<^f$ has positive cone $f^{-1}Pf$.   This conjugation action on positive cones defines a action of $G$ by homeomorphisms of $\mathrm{LO}(G)$.

A \textit{left-circular ordering} is a function $c:G^3 \rightarrow \{0 , \pm1 \}$ satisfying:
	\begin{enumerate}
		\item If $(g_1, g_2, g_3) \in G^3$ then $c(g_1, g_2, g_3) = 0$ if and only if $\{g_1, g_2, g_3\}$ are not all distinct;
		\item For all $g_1, g_2, g_3, g_4 \in G$ the function $c$ satisfies 
		\[ c(g_1, g_2, g_3) - c(g_1, g_2, g_4) + c(g_1, g_3, g_4)-c(g_2, g_3, g_4) = 0;
		\]
		and
		\item For all $g, g_1, g_2, g_3 \in G$ we have 
		\[ c(g_1, g_2, g_3) = c(gg_1, gg_2, gg_3).
		\]
	\end{enumerate}
	If $G$ admits such a map, then $G$ is called {\it left-circularly orderable} or simply {\it circularly orderable}, and the pair $(G, c)$ will be called a {\it circularly ordered group}.

The set $\mathrm{CO}(G)$ of all such functions is a closed subspace of $\{0, \pm 1\}^{G^3}$, and as such it is compact.  The topology inherited by $\mathrm{CO}(G)$ has as subbasis the sets 
\[ U_{(g_1, g_2, g_3)}^i = \{ c \in \mathrm{CO}(G) \mid c(g_1, g_2, g_3) = i \},
\]
where $i \in \{0, \pm 1\}$ and $(g_1, g_2, g_3)$ ranges over all triples in $G^3$.  

Every left-ordering $<$ of a group $G$ defines a circular ordering $c_<$ according to the prescription $c_<(g_1, g_2,g_3) = \mathrm{sign}(\s)$ where $\s$ is the unique permutation such that $g_{\s(1)}<g_{\s(2)}<g_{\s(3)}$.  We call the circular orderings that arise in this way \textit{secret left-orderings}, and they are precisely the image of the natural embedding $\iota : \mathrm{LO}(G) \rightarrow \mathrm{CO}(G)$ given by $\iota(<) = c_<$ (injectivity follows from the fact that $P = \{ g\in G \mid c_<(g^{-1}, id, g) = 1 \}$ is the positive cone of the left-ordering that determines $c_<$, so that $c_{<_1} = c_{<_2}$ implies $<_1 = <_2$).  There is a natural right action of $G$ on $\mathrm{CO}(G)$ by homeomorphisms.  If $c$ is a circular ordering and $g \in G$, the circular ordering $c \cdot g$ is given by $(c \cdot g)(g_1, g_2, g_3) = c(g_1g^{-1}, g_2g^{-1}, g_3g^{-1})$ for all $ c\in \mathrm{CO}(G)$ and $g \in G$, whose restriction to $\iota(\mathrm{LO}(G))$ is the action on $\mathrm{LO}(G)$ already described above.

The cases where $\mathrm{LO}(G)$ and $\mathrm{CO}(G)$ are finite are well-understood.  The space $\mathrm{LO}(G)$ is finite if and only if $G$ admits a rational series 
\[ T_0= \{id\} \triangleleft T_1 \triangleleft \dots \triangleleft T_{k-1} \triangleleft T_k =G 
\]
whose quotients $T_i/T_{i-1}$ are rank one abelian, and such that $T_i/T_{i-2}$ is not bi-orderable for any $i = 2, \dots ,k$.  Such groups are called \textit{Tararin groups}, and the terms $T_i$ appearing in the rational series are all absolutely convex subgroups (i.e. convex in every left-ordering of $T_k)$, thus yielding precisely $2^k$ left-orderings on the Tararin group $T_k$ \cite{Ta}.  The space $\mathrm{CO}(G)$ is finite if and only if $G \cong T_k \ltimes \mathbb{Z}/n\mathbb{Z}$ where $T_k$ is a Tararin group, $n$ is even, and the action of the generator on $T_k/T_{k-1}$ is by inversion \cite{CMR}.

\section{Natural and unnatural orderings}
\label{not natural}

Before introducing natural and unnatural orderings of groups, we begin by recalling a standard way of constructing a left-ordering of a group of automorphisms of a totally ordered set,  and observe some of the basic structure of the resulting ordering.

\begin{construction}
\label{standard_construction}
Suppose that $(\Omega, <)$ is a totally ordered set, and let $\mathrm{Aut}(\Omega, <)$ denote the set of order-preserving automorphisms of $(\Omega, <)$. Fix a well-ordering $\prec$ of $\Omega$.  Define the positive cone $P$ of a left-ordering of $\mathrm{Aut}(\Omega, <)$ as follows:  Given $f \in \mathrm{Aut}(\Omega, <)$, set $x_{f} = \min_{\prec}\{ x \in \Omega \mid f(x) \neq x\}$.  Declare $f \in P$ if and only if $f(x_{f}) >x_{f}$.
\end{construction}

Note that the positive cone $P$ in the construction above depends on the choice of well-ordering $\prec$ of the set $\Omega$.  Nonetheless, we leave this dependency implicit and do not choose a notation such as $P_{\prec}$ or $P(\prec)$, as this notation is typically reserved in the literature to denote the positive cone of the left-invariant ordering $\prec$ of a left-ordered group.

\begin{proposition}
\label{convex subgroups}
Using the notation of Construction \ref{standard_construction}, suppose that $S \subset \Omega$ satisfies $x \in S$ and $y \prec x$ implies $y \in S$.  Then the stabilizer 
\[ C_S = \{ \phi \in \mathrm{Aut}(\Omega, <) \mid \phi(s) =s \mbox{ for all $s \in S$} \}
\]
 is convex in the left-ordering of $ \mathrm{Aut}(\Omega, <)$ determined by the positive cone $P$.
\end{proposition}
\begin{proof}
Let $f\in  \mathrm{Aut}(\Omega, <)$ and $g\in C_S$ such that $1<_Pf<_Pg$, where $<_P$ is the left-ordering defined by the positive cone $P$. We have that $x_f< f(x_f)$ and $x_{f^{-1}g}<f^{-1}g(x_{f^{-1}g})$. If $f\in C_S$ then we have nothing to prove. Assume that $f\notin C_S$. Then there exists $s \in S$ such that $f(s) \neq s$, and since $S$ is downward closed this implies $x_f\in S$. Since $g\in C_S$, $f^{-1}g(s)=f^{-1}(s)$ for any $s\in S$. So, if $s \prec x_f$ then $f^{-1}g(s)=s$ which implies that $x_f=x_{f^{-1}g}$. We have that $f^{-1}g(x_f)=f^{-1}(x_f)<x_f$ which is a contradiction to the fact that $x_{f^{-1}g}<f^{-1}g(x_{f^{-1}g})$.
\end{proof}

\begin{definition}
Call a left-ordering of $\mathrm{Homeo}_+(\mathbb{R})$ \emph{natural} if it arises from an application of Construction \ref{standard_construction}.  Given a group $G$, call a left-ordering $<$ of $G$ \emph{natural} if there exists an injective representation $\rho:G \rightarrow \mathrm{Homeo}_+(\mathbb{R})$ and a natural ordering $\prec$ of $\mathrm{Homeo}_+(\mathbb{R})$ such that $< = \prec^{\rho}$.  If a left-ordering of $G$ is not natural, it is called \emph{unnatural}.
\end{definition}

It is a standard result that if $G$ is countable, then every left-ordering $<$ of $G$ is natural.  Indeed, the dynamical realization of the ordering provides an injective representation $\rho :G \rightarrow  \mathrm{Homeo}_+(\mathbb{R})$ such that $g<h$ if and only if $\rho(g)(0) < \rho(h)(0)$.  As such, an application of Construction \ref{standard_construction} using a well-ordering of $\mathbb{R}$ for which $0$ is the smallest element will give a natural ordering $\prec$ of $ \mathrm{Homeo}_+(\mathbb{R})$ such that $\prec^\rho =<$. 

On the other hand, it is easy to construct left-orderable groups having no natural orderings whatsoever.  For example, consider a torsion-free abelian group $A$ of cardinality larger than the cardinality of $\mathrm{Homeo}_+(\mathbb{R})$.  The group $A$ is left-orderable since it is torsion-free abelian, but evidently no orderings are natural since an injective representation $\rho :A \rightarrow  \mathrm{Homeo}_+(\mathbb{R})$ cannot exist for cardinality reasons.  There are, however, groups that admit no action on $\mathbb{R}$ for reasons other than cardinality. 

For example, the group generated by $\{ a_s \mid s \in \mathbb{R} \}$ with relations 
\[ a_s a_r a_s^{-1} = a_r^{-1} \mbox{ whenever $r<s$}
\] 
is left-orderable, yet admits no action on $\mathbb{R}$ by homeomorphisms, and thus no natural orderings---yet it has cardinality $2^{\aleph_0}$ (see \cite[Example 2.2.4]{DNR} and \cite[Proposition 5.2]{Man1}, this second paper also has many other great examples).

In contrast to these results, we will show that the free group on $2^{\aleph_0}$ generators admits both natural and unnatural orderings. 

\begin{proposition}
\label{not_natural}
Let $F(X)$ denote the free group with free generating set $X = \{x_r\}_{r \in \mathbb{R}}$.  For every $S \subset \mathbb{R}$ satisfying 
\[ x \in S \mbox{ and } y<x \mbox{ implies } y \in S \hspace{1em} (*)
\]
let $X_S = \{ x_r \}_{r \in S}$ and let $F(X_S)$ denote the subgroup of $F(X)$ generated by $X_S$.  Then $F(X)$ admits a left-ordering relative to which $F(X_S)$ is convex for every $S \subset \mathbb{R}$ satisfying $(*)$.
\end{proposition}
\begin{proof}
By Sunic \cite[Corollary 20]{ADS}), every $F(X_S)$ is relatively convex since free factors are relatively convex.  As such, for each $S \subset \mathbb{R}$ satisfying $(*)$, we may fix a choice of positive cone $P_S \subset F(X)$ corresponding to a left-ordering of $F(X)$ that has $F(X_S)$ as a convex subgroup.  Define $P \subset F(X)$ as follows.

Given $y \in F(X)$, if $y = \prod_{i=1}^n x_{r_i}$ is a reduced word, then set $r_y = \max\{r_i \}$.  Set $S_y = \{ x \in \mathbb{R} \mid x < r_y\}$ and declare $y \in P$ if and only if $y \in P_{S_y}$.  It is straightforward to check that $P$ defines the positive cone of a left-ordering relative to which $F(X_S)$ is convex for every $S \subset \mathbb{R}$ satisfying $(*)$.
\end{proof}

\begin{proposition}
The free group $F(X)$ admits both natural and unnatural orderings.
\end{proposition}
\begin{proof}
We first show that the ordering $<$ of $F(X)$ constructed in Proposition \ref{not_natural} is not natural.  To this end, suppose that $\rho:F(X) \rightarrow \mathrm{Homeo}_+(\mathbb{R})$ is an injective representation and that $\prec$ is a natural ordering of $\mathrm{Homeo}_+(\mathbb{R})$ such that $< = \prec^{\rho}$. Set 
\[x_0 = \min \{ x \in \mathbb{R} \mid \exists g \in F(X) \mbox{ such that } \rho(g)(x_0) \neq x_0\},
\] 
where the minimum is taken with respect to the well-ordering of $\mathbb{R}$ used to define the natural ordering $\prec$ of $\mathrm{Homeo}_+(\mathbb{R})$.
Let $C \subset F(X)$ denote the stabilizer of $x_0$, that is $C = \{ g \in F(X) \mid \rho(g)(x_0) = x_0 \}$, and note that by Proposition \ref{convex subgroups}, $C$ must be convex relative to the ordering $<$ of $F(X)$.

Since convex subgroups are ordered by inclusion, and there is no greatest convex subgroup of $F(X)$ with respect to the ordering $<$ of the form $F(X_S)$, we may choose a subset $S \subset \mathbb{R}$ such that $C \subset F(X_S)$.  For ease of notation, we will write $D$ in place of the convex subgroup $F(X_S)$.

Set 
\[ I = \{ x \in \mathbb{R} \mid \exists d, d' \in D \mbox{ such that } \rho(d)(x_0) <x < \rho(d')(x_0) \},
\]
and note that $I$ is an open interval, since  
\[ I = \bigcup_{\substack{d, d' \in D \\ d<d'}} (\rho(d)(x_0), \rho(d')(x_0)).
\]

Next, suppose that $g \notin D$, we check that $\rho(g)(I) \cap I =\emptyset$.  Assume that $g>1$,  and observe that if $gd<1$, then $d<g^{-1}<1$ implies $g \in D$ by convexity, so that we must in fact have $gd>1$ for all $d \in D$. We use this fact in what follows.

First, note it suffices to prove $\rho(gd)(x_0) > \rho(d')(x_0)$ for all $d, d' \in D$.  So suppose the contrary, say $\rho(gd)(x_0) \leq \rho(d')(x_0)$.  If $\rho(gd)(x_0) = \rho(d')(x_0)$ then $\rho((d')^{-1}gd)(x_0) = x_0$ so that $(d')^{-1}gd \in C \subset D$, a contradiction since $g \notin D$.  If $\rho(gd)(x_0) < \rho(d')(x_0)$, then since $gd>1$ we know $x_0< \rho(gd)(x_0) < \rho(d')(x_0)$, so that $1< gd <d'$ and hence $gd \in D$ by convexity, again contradicting $g \notin D$.  

We conclude that $\rho(g)(I) \cap \rho(h)(I) =\emptyset$ whenever $gD$ and $hD$ are distinct cosets.  Recalling that $D =F(X_S)$, the collection $\{x_r\}_{r \notin S}$ is an uncountable collection of distinct coset representatives, and thus $\{ \rho(x_r)(I) \}_{r \notin S}$ is an uncountable collection of disjoint open intervals in $\mathbb{R}$, which is absurd.  Thus the ordering $<$ of $F(X)$ is unnatural.

On the other hand, $F(X)$ may be embedded as a subgroup of $\mathrm{Homeo}_+(\mathbb{R})$ \cite[Theorem 1]{BK}. So it admits plenty of natural orderings.
\end{proof}

The same result holds for groups admitting circular orderings and actions on $S^1$. We can analogously define natural and unnatural circular orderings of a group $G$, and arrive at:

\begin{proposition}
There exists a group $G$ admitting both natural and unnatural circular orderings.
\end{proposition}
\begin{proof}
If $H$ is a circularly orderable group, any lexicographic circular ordering of $G = F(X) \times H$ is unnatural if we equip $F(X)$ with an unnatural left-ordering.  Yet there are choices of $H$, such as $H = \mathbb{Z}/n\mathbb{Z}$ for any $n \geq 2$, for which the product $F(X) \times H$ clearly embeds in $\mathrm{Homeo}_+(S^1)$. 
\end{proof}

Crist\'{o}bal Rivas also pointed out to the authors in private communication that there are bi-orderings of $\mathrm{PL}_+([0,1])$ that behave in a similar manner, owing to the fact that they admit $2^{\aleph_0}$ convex jumps.  These bi-orderings arise, roughly, by ordering the elements of $\mathrm{PL}_+([0,1])$ according to their derivatives on the leftmost (or rightmost) interval on which they are not the identity.  This is therefore another example of a group admitting both natural and unnatural orderings.

\section{Generalized semiconjugacy}
\label{semiconjugacy_section}

Traditional notions of semiconjugacy relate strictly to actions on $\mathbb{R}$ or on $S^1$, from which one can study left or circular orderings via semiconjugacy of their dynamical realizations. 

This means that for a countable group $G$, whose circular and left-orderings all arise from particularly nice actions on $S^1$ and $\mathbb{R}$, the notion of semiconjugacy appearing in the literature is adequate to analyze $\mathrm{LO}(G)$ (or $\mathrm{CO}(G)$) and its equivalence classes of orderings up to semiconjugacy.  More generally, if all circular and left-orderings of a group $G$ are natural, then semiconjugacy of actions on $S^1$ or $\mathbb{R}$ is potentially an adequate notion of semiconjugacy for a study of the orderings of $G$.  At the other extreme, for a circularly orderable or left-orderable group with no actions on $S^1$ or $\mathbb{R}$ whatsoever, it would be reasonable to conclude that traditional semiconjugacy of actions is not a relevant tool for an analysis of the group's orderings.

However, the existence of groups admitting both natural and unnatural orderings points to a need for a generalized notion of semiconjugacy in order to give a satisfactory analysis of the space of orderings of an arbitrary group.  The definitions and proofs mirror the classical ones, and proceed more or less as expected.   There are, however, a few hiccups along the way that require a bit of care (e.g. see Example \ref{not_sequential}).

\subsection{Semiconjugacy of left and circular orderings}
Throughout this section, whenever $(S, <)$ is an ordered set we denote the Dedekind--MacNeille completion of $S$ by $D(S, <)$, which we will think of as being constructed from cuts $(U, V)$.  The natural ordering of $D(S, <)$ will also be denoted by $<$ whenever no confusion arises from doing so. When there is no risk of confusion, we will write $D(S)$ in place of $D(S, <)$ and the image of $s \in S$ under the natural inclusion $S \hookrightarrow D(S)$ will be denoted simply by $s$.

The group of order-preserving automorphisms of a totally ordered set $(S, <)$ will be denoted by $\Aut(S, <)$.  Whenever there is a representation $\rho : G \rightarrow \Aut(S, <)$, we will denote the unique extension by $\overline{\rho} : G \rightarrow \Aut(D(S), <)$.

  Recall that if $(G, <)$ is a left-ordered group, then a subset $X \subset G$ is \textit{cofinal} if for every $g \in G$ there exists $x, y \in X$ such that $y < g < x$.  We say that an element $g \in G$ is cofinal if the cyclic subgroup $\langle g \rangle$ is cofinal.  We recall the following standard construction.

\begin{construction}
\label{lift}
Let $(G,c)$ be a circularly ordered group.  Let $\widetilde{G}_c$ denote the central extension 
of $G$ by $\Z$ constructed from the set $G\times \Z$ which we equip with a multiplication given by
$(g,n)(h,m)=(gh,n+m+f_c(g,h))$, where
$$ f_c(g,h)=\left\{\begin{array}{cl} 0 & \text{if $g=id$ or $h=id$ or $c( id, g, gh)=1$}
\\ 1 &  \text{if $gh=id$ $(g\not=id)$ or $c(id,gh,g)=1$.  }  \end{array}
\right.$$ 
Define the positive cone of a left-ordering $<_c$ by $$P=\{(g,n)\mid n\geq 0\} \setminus \{ (id, 0) \}.$$  The central element $z_c = (id, 1)$ is positive and cofinal with respect to $<_c$.
\end{construction}

\begin{definition}
\label{circ_semiconjugacy}
Let $c$, $d$ be circular orderings of a group $G$.  
We say that $c$ is \emph{semiconjugate} to $d$ if there exist a nondecreasing map $f: D(\widetilde{G}_c) \rightarrow D(\widetilde{G}_d)$ $($i.e. $x \leq_c y$ implies $f(x) \leq_d f(y))$ and an equivalence of extensions $\phi: \widetilde{G}_c \rightarrow \widetilde{G}_d$ such that $f(h(\alpha)) = \phi(h)(f(\alpha))$ for all $\alpha \in D(\widetilde{G}_c)$ and $h \in \widetilde{G}_c$.
\end{definition}

\begin{remark} First a remark on notation.  When we write $h(\alpha)$, we are implicitly identifying an element $h \in \widetilde{G}_c$ with the order-preserving automorphism of $D(\widetilde{G}_c)$ that arises as the unique extension of left-multiplication by $h$ on $\widetilde{G}_c$.  Similarly for $\phi(h)(f(\alpha))$. Second, if we let $z_c \in \widetilde{G}_c$ and $z_d \in \widetilde{G}_d$ denote the central, cofinal positive elements arising from Construction \ref{lift}, then $\phi$ satisfies $\phi(z_c) = z_d$.  This definition therefore captures the notion of $f$ being a monotone map (i.e. satisfying $f(z_c(\alpha)) = z_d f(\alpha)$, analogous to the traditional definition of semiconjugacy) by setting $h = z_c$ in $f(h(\alpha)) = \phi(h)(f(\alpha))$.
\end{remark}

\begin{proposition}
Semiconjugacy of circularly ordered groups is an equivalence relation.
\end{proposition}
\begin{proof}
That semiconjugacy is a reflexive relation follows from the definition, as we can take $f$ and $\phi$ in the definition above to be the identity.  Transitivity follows just as easily, by noting that the composition of equivalences of extensions yields an equivalence of extensions, and the that the composition of nondecreasing maps yields a nondecreasing map.

For symmetry, suppose that $c$ and $d$ are semiconjugate circular orderings of $G$, and that $f$, $\phi$ are as in Definition \ref{circ_semiconjugacy}.  Our approach is standard, we follow \cite{Man2} and include details for the sake of completeness.  Define $f' : D(\widetilde{G}_d) \rightarrow D(\widetilde{G}_c)$ by
\[ f'(\beta) = \sup_{\leq_c} \{ \alpha \in D(\widetilde{G}_c) \mid f(\alpha) \leq_d \beta \}.
\]
  
From the definition, one can verify that $f$ is nondecreasing.  

We will check that $f'h(\beta) = \phi^{-1}(h) f'(\beta)$ for all $\beta \in D(\widetilde{G}_d)$ and $h \in \widetilde{G}_d$.  We begin with
\[ f' h(\beta) = \sup_{\leq_c} \{  \alpha \in D(\widetilde{G}_c) \mid f(\alpha) \leq_d h(\beta) \} = \sup_{\leq_c} \{  \phi^{-1}(h)(\alpha) \in D(\widetilde{G}_c) \mid f( \phi^{-1}(h)(\alpha)) \leq_d h(\beta) \}.
\]
Now note that $ f(\phi^{-1}(h)(\alpha))  = \phi(\phi^{-1}(h))f(\alpha) = hf(\alpha)$, so $f( \phi^{-1}(h)(\alpha)) \leq_d h(\beta)$ is equivalent to $hf(\alpha) \leq_d h(\beta)$, or $f(\alpha) \leq_d \beta$ since $h$ is order-preserving.  Thus 
\[ f' h(\beta) = \sup_{\leq_c} \{  \phi^{-1}(h)(\alpha) \in D(\widetilde{G}_c) \mid f(\alpha) \leq_d \beta \} = \phi^{-1}(h) f'(\beta).
\]
\end{proof}

\subsection{Rotation and translation numbers}

One can define rotation numbers and translation numbers in the expected ways.  First, some notation:  If $(G, c)$ is a circularly ordered group and $g \in \widetilde{G}_c$, then $[g]_c$ will be used to denote the unique integer such that $z_c^{[g]_c} \leq_c g <_c z_c^{[g]_c+1}$. 

\begin{definition}
\label{rot def}
Suppose that $G$ is a group, and let $c$ be a circular ordering of $G$.  For each $h \in \widetilde{G}_c$, define 
\[ \widetilde{\mathrm{rot}}_c(h) = \lim_{n \to \infty} \frac{[h^n]_c}{n}.
\]
Given $g \in G$, call $\widetilde{g} \in \widetilde{G}_c$ a lift of $g$ if it maps to $g$ under the natural quotient, that is, $\widetilde{g}\langle z_c \rangle = g$.  

For every $g \in G$, define the rotation number of $g$ with respect to the circular ordering $c$ to be $\mathrm{rot}_c(g) = \widetilde{\mathrm{rot}}_c(\widetilde{g}) \mod \mathbb{Z}$, where $\widetilde{g}$ is any lift of $g$.  For any two elements $g, h \in G$, choose lifts $\widetilde{g}, \widetilde{h} \in \widetilde{G}$ and define
\[ \tau_c(g, h) = \widetilde{\mathrm{rot}}_c(\widetilde{g}\widetilde{h}) - \widetilde{\mathrm{rot}}_c(\widetilde{g}) - \widetilde{\mathrm{rot}}_c(\widetilde{h}).
\]
\end{definition}

The limit in the previous definition always exists by an application of Fekete's Lemma, and one can verify that both $\mathrm{rot}_c(g)$ and $\tau_c(g,h)$ are independent of choices of lifts.  Moreover, one can show they are invariant under conjugation by elements of $G$---that is, the following proposition holds.

\begin{proposition}
\label{conj_inv}
Suppose that $G$ is a group and $c$ is a circular ordering of $G$.  The following properties hold for all $f,g,h \in G$:
\begin{enumerate}
\item $\mathrm{rot}_c(g) = \mathrm{rot}_c(hgh^{-1})$
\item $\tau_c(g,h) = \tau_c(fgf^{-1}, fhf^{-1})$.
\end{enumerate} 
\end{proposition}

These quantities are enough to characterize semiconjugacy of circular orderings.  The following is a standard result due to Matsumoto \cite{Mat}, though we follow the approach of \cite{Man2} in order to circumvent bounded cohomology and keep this exposition self-contained.

\begin{proposition}
Let $G$ be a group.  Two circular orderings $c$ and $d$ of $G$ are semiconjugate if and only if 
\label{semi classes}
\begin{enumerate}
\item There exists a generating set $\{s_i \}_{i \in I}$ of $G$ such that $\mathrm{rot}_c(s_i) = \mathrm{rot}_d(s_i)$ for all $i \in I$.
\item For all $g, h \in G$ we have $\tau_c(g, h) = \tau_d(g,h)$.
\end{enumerate}
\end{proposition}
\begin{proof}
	First suppose that $c$ and $d$ satisfy conditions (1) and (2).  For each $s_i$, choose lifts $u_i \in \widetilde{G}_c$ and $v_i \in \widetilde{G}_d$ such that $\widetilde{\mathrm{rot}}_c(u_i) = \widetilde{\mathrm{rot}}_d(v_i)$.  This is possible since $\mathrm{rot}_c(s_i) = \mathrm{rot}_d(s_i)$, so $\widetilde{\mathrm{rot}}_c(u_i)$ and $\widetilde{\mathrm{rot}}_d(v_i)$ differ (\textit{a priori}) by an integer quantity which can be corrected by multiplying by appropriate powers of central elements $z_c$ and $z_d$.
	
	Then for any finite word in the generators $\{s_i\}_{i \in I}$, say $s_{i_1} \cdots s_{i_n}$, the corresponding products of lifts $u_{i_1} \cdots u_{i_n}$ and $v_{i_1} \cdots v_{i_n}$ satisfy
	\[ \widetilde{\mathrm{rot}}_c(u_{i_1} \cdots u_{i_n}) =\widetilde{\mathrm{rot}}_d(v_{i_1} \cdots v_{i_n})
	\]
	which we can show by inductively applying (2).
	
	Now note that $\{u_i \}_{i \in I} \cup \{ z_c \}$ is a generating set for $\widetilde{G}_c$, and $\{ v_i \}_{i \in I} \cup \{ z_d \}$ is a generating set for $\widetilde{G}_d$.  We can define an equivalence of extensions $\phi: \widetilde{G}_c \rightarrow \widetilde{G}_d$ by $\phi(u_i) = v_i$ for all $i \in I$ and $\phi(z_c) = z_d$.  That this defines a homomorphism follows from the observation that $u_{i_1} \cdots u_{i_n} z_c^{\ell} = id$ implies that $s_{i_1} \cdots s_{i_n} = id \in G$, and thus $v_{i_1} \cdots v_{i_n}$ is a power of $z_d$.  Therefore since $ \widetilde{\mathrm{rot}}_c(u_{i_1} \cdots u_{i_n}) = - \ell$, then $\widetilde{\mathrm{rot}}_d(v_{i_1} \cdots v_{i_n}) = -\ell$ so that we must have $v_{i_1} \cdots v_{i_n} z_d^{\ell} = id$.  That $\phi$ defines an equivalence of extensions is obvious from its definition.   
	
	Next we construct $f: D(\widetilde{G}_c) \rightarrow D(\widetilde{G}_d)$.  To do this, note that the containment $\widetilde{G}_c \subset D(\widetilde{G}_c)$ allows us to apply the map $\phi$ to certain elements of $D(\widetilde{G}_c)$, so that we may define
	\[ f( \alpha) = \sup_{\leq_d} \{ \phi(g) \mid g \in \widetilde{G}_c  \mbox{ and } g \leq_c \alpha \}.
	\]
	We argue that this supremum is finite, so that $f$ is well-defined. To see this, observe that if $g \in \widetilde{G}_c$, we have $[g]_c \leq \widetilde{\mathrm{rot}}_c(g) \leq [g]_c+1$ (and the same holds for elements of $\widetilde{G}_d$ and $[\cdot]_d$).  Now given $\alpha \in \widetilde{G}_c$, choose $k$ such that $\alpha <_c z_c^k$.  Then if $g = u_{i_1} \cdots u_{i_n} z_c^{\ell} \in \widetilde{G}_c$ satisfies $g < \alpha$, then $u_{i_1} \cdots u_{i_n} < z^{k-\ell}$ and so $\widetilde{\mathrm{rot}}_c(u_{i_1} \cdots u_{i_n}) <k - \ell +1$.  But then $\widetilde{\mathrm{rot}}_d(v_{i_1} \cdots v_{i_n}) < k-\ell+1$, and so $[v_{i_1} \cdots v_{i_n}]_d +1 \leq k-\ell+1$.  It follows that  $v_{i_1} \cdots v_{i_n} < z_d^{k-\ell+3}$.  Thus $\phi(u_{i_1} \cdots u_{i_n} z_c^{\ell}) =v_{i_1} \cdots v_{i_n}z_d^{\ell}<z_d^{k+3}$, and so all such elements are bounded in $\widetilde{G}_d$.  Thus $f$ is well-defined, and note that it is plainly nondecreasing by definition.

 Fix  $h \in \widetilde{G}_c$ and $\alpha \in D(\widetilde{G}_c)$, we will next verify that $f(h(\alpha)) = \phi(h)(f(\alpha))$.  We compute 
	\[ f(h(\alpha)) = \sup \{ \phi(g) \mid g \leq_c h(\alpha) \} = \sup \{ \phi(g) \mid h^{-1}g \leq _c\alpha \}
	\]
	and note that the set of $\phi(g)$ with $h^{-1}g \leq_c \alpha$ is the same as the set of elements $\phi(h)\phi(h^{-1}g)$ with $h^{-1}g \leq_c \alpha$.    Thus
	\[ f(h(\alpha)) = \sup_{\leq_d}\{ \phi(h) \phi(h^{-1}g) \mid h^{-1}g \leq_c \alpha \} = \sup_{\leq_d} \{ \phi(h) \phi(s) \mid s \leq \alpha \} = \phi(h)(f(\alpha)).
	\]  Thus conditions (1) and (2) imply that there is a semiconjugacy between the circular orderings $c$ and $d$.
	
	On the other hand, suppose that $c$ and $d$ are semiconjugate with $f : \widetilde{G}_c \rightarrow \widetilde{G}_d$ a nondecreasing map, and $\phi$ an equivalence of extensions that demonstrate the semiconjugacy.  Then $\phi(g)f(\alpha) = f((\alpha))$ for all $g \in \widetilde{G}_c$ and all $\alpha \in D(\widetilde{G}_d)$.  Now if $n, k$ are any integers, from $z_c^k \leq_c g <_c z_c^{k+1}$ applying $f$ yields $z_d^k f(id) \leq_d \phi(g) f(id) \leq_d z_d^{k+1}f(id)$.   Choose $\ell$ such that $z_d^{\ell} \leq_d f(id) <_d z_d^{\ell +1}$, then $z_d^{k+1} f(id) <_d z_d^{\ell +k + 2}$.  Now if $z_d^{k+2} \leq_d \phi(g)$ then $z_d^{k+\ell+2} \leq_d \phi(g)z_d^{\ell}$, further we can apply $\phi(g)$ on the left of $z_d^{\ell} \leq_d f(id) <_d z_d^{\ell +1}$ to arrive at $\phi(g)z_d^{\ell} \leq_d \phi(g) f(id)$.  Stringing these inequalities together, we have 
	$$\phi(g) f(id) \leq_d z_d^{k+1}f(id) <_d z_d^{\ell +k + 2} \leq_d \phi(g)z_d^{\ell}  \leq_d \phi(g) f(id).$$

This is a contradiction, so one concludes $ \phi(g)  <_d z_d^{k+2}$.  Similarly one can show that $z_d^{k-1}$ is a lower bound for $\phi(g)$, yielding $z_d^{k-1}  \leq_d \phi(g)  <_d z_d^{k+2}$.   We conclude that $| [g]_c - [\phi(g)]_d| \leq 1$ for all $g \in \widetilde{G}_c$, from which it follows that $\widetilde{\mathrm{rot}}_c(g) = \widetilde{\mathrm{rot}}_d(\phi(g))$.  
	
	Now suppose $g \in G$ and $\widetilde{g}$ is a lift of $g$ in $\widetilde{G}_c$, so then $\phi(\widetilde{g})$ is a lift of $g$ in $\widetilde{G}_d$.  From the previous paragraph, $\widetilde{\mathrm{rot}}_c(\widetilde{g}) = \widetilde{\mathrm{rot}}_d(\phi(\widetilde{g}))$ and it follows that  $\mathrm{rot}_c(g) = \mathrm{rot}_d(g)$.  Similarly $\tau_c(g,h) = \tau_d(g,h)$ for all $g, h \in G$.

\end{proof}

We note that with only slight modifications to this proof, it is easy to show that two circular orderings $c$ and $d$ of $G$ are semiconjugate if and only if $\mathrm{rot}_c = \mathrm{rot}_d$ and $\tau_c(g, h) = \tau_d(g,h)$ as maps $G \rightarrow \mathbb{R}/\mathbb{Z}$ and $G \times G \rightarrow \mathbb{R}$.  We also note the following corollary, indicating that our definitions agree with the classical constructions.

\begin{corollary}
Two circular orderings $c_1$ and $c_2$ of a group $G$ are semiconjugate if and only if their dynamical realizations $\rho_1, \rho_2 : G \rightarrow \mathrm{Homeo}_+(S^1)$ are semiconjugate in the classical sense.
\end{corollary}
\begin{proof}
One need only verify that our definition of $\widetilde{\mathrm{rot}}_{c_i}$ agrees with the classical definition of translation number arising from the dynamical realizations $\rho_1, \rho_2 : G \rightarrow \mathrm{Homeo}_+(S^1)$. 
\end{proof}

With Proposition \ref{semi classes} in hand, we are able to show that generalized semiconjugacy behaves as one would expect with respect to common constructions, such as convex subgroups and secret left-orderings.  We first require a technical lemma.

\begin{lemma}
\label{open set}
Let $(G,c)$ be a circularly ordered group.  Fix $g \in G$ and let $\widetilde{g} = (g, 0) \in \widetilde{G}_c$.  Fix another circular ordering $d$ of $G$ and consider the element $h = (g, 0) \in \widetilde{G}_d$.  
For each $n \geq 0$, if $d$ lies in the open neighbourhood 
\[U_n= \bigcap_{i=1}^n U_{(id, g^i, g^{i+1})}^{c(id, g^i, g^{i+1})}  \subset \mathrm{CO}(G)
\]  then $[h^n]_d = [\widetilde{g}^n]_c$.
\end{lemma}
\begin{proof} 
The condition that $[h^n]_{d} = [\widetilde{g}^n]_c$ holds if and only if the inequalities
\[ z_d^{[\widetilde{g}^n]_c} \leq_d h^n <_d z_d^{[\widetilde{g}^n]_c+1}
\]
hold in $\widetilde{G}_d$.  These inequalities are a consequence of a finite number of conditions on $d$; specifically we require
\[ id \leq_d h^n z_d^{-[\widetilde{g}^n]_c} \mbox{ and } id <_d h^{-n}z_d^{[\widetilde{g}^n]_c+1}.
\]
We calculate that for $n>0$
\[ h^n z^{-[\widetilde{g}^n]_c} = (g,0)^n(id, -[\widetilde{g}^n]_c) = \left( g^n, \, \sum_{i=1}^{n-1}f_d(g^i, g) -[\widetilde{g}^n]_c \right)
\]
and so $id \leq_d h^n z^{-[\widetilde{g}^n]_c}$ if and only if $\sum_{i=1}^{n-1}f_d(g^i, g) -[\widetilde{g}^n]_c \geq 0$.  Since we know 
\[ z_c^{[\widetilde{g}^n]_c} \leq_c (g,0)^n <_c z_c^{[\widetilde{g}^n]_c+1}
\]
holds in $\widetilde{G}_c$, it follows that $\sum_{i=1}^{n-1}f_c(g^i, g) -[\widetilde{g}^n]_c \geq 0$.  Thus if $f_d(g^i, g) = f_c(g^i, g)$ for $i=1, \ldots, n-1$ then $id \leq_d h^n z_d^{-[\widetilde{g}^n]_c}$, and we conclude $[h^n]_{d} \geq [\widetilde{g}^n]_c$.  If we additionally insist that $f_d(g^n, g) = f_c(g^n, g)$ then we get the second inequality $ id <_d h^{-n}z^{[\widetilde{g}^n]_c+1}$ and so conclude $[h^n]_{d} = [\widetilde{g}^n]_c$.

It follows that since the equalities $f_d(g^i, g) = f_c(g^i, g)$ for $i = 1, \ldots, n-1$ and $f_d(g^n, g) = f_c(g^n, g)$ hold whenever
\[ d \in \bigcap_{i=1}^n U_{(id, g^i, g^{i+1})}^{c(id, g^i, g^{i+1})} ,
\]
then for all such $d$, we have $[h^n]_{d} = [\widetilde{g}^n]_c$.
\end{proof}

Recall that a subgroup $C$ in a circularly ordered group $(G,c)$ is \textit{convex relative to $c$} if the left cosets $G/C$ inherit a natural circular ordering $\overline{c}$, defined by 
\[ \overline{c}(g_1C, g_2C, g_3C) = c(g_1, g_2, g_3)
\]
whenever the cosets $g_1C, g_2C$ and $g_3C$ are distinct.  Alternatively, one can define $C$ to be convex if for every $g \in G \setminus C$, $f \in G$, and $c_1, c_2 \in C$ the implication 
\[[c(c_1, g, c_2) = 1 \mbox{ and } c(c_2, f, c_1) = 1 ] \Rightarrow f \in C
\]
holds.  That these notions are equivalent appears as \cite[Lemma 5.1]{CG}.
\begin{proposition}
\label{prop convex}
Suppose that $C$ is convex relative to the circular orderings $c$ and $d$ of $G$. If 
 the quotient orderings $\overline{c}$ and $\overline{d}$ of the set of cosets $G/C$ induced by $c$ and $d$ agree, then $c$ and $d$ are semiconjugate.  
\end{proposition}
\begin{proof}
If $C$ is convex in $(G,c)$, it admits a canonical positive cone 
\[ P = \{ h \in C \mid c(id, h, g) = 1 \mbox{ for some } g \in G \setminus C \}, 
\]
see \cite[Lemma 5.2]{CG} for details. Define a function $\eta:C \rightarrow \{0, 1 \}$ by $\eta(id) =0$ and if $h \neq id$ then $\eta(h) =0$ if and only if $h \in P$.  This yields a function satisfying $f_c(g,h) = \eta(g)-\eta(gh)+\eta(h)$ for all $g, h \in C$, so that $f_c$ is a coboundary when restricted to $C$, and therefore $\widetilde{C}_c$ is a split extension.   Consequently there is an embedding 
\[\phi:C \hookrightarrow  \widetilde{C}_c \hookrightarrow \widetilde{G}_c\]
that is given by $\phi(h) = (h, -\eta(h))$, and whose image turns out to be a convex subgroup of the left-ordering $<_c$ of $\widetilde{G}_c$ (\cite[Lemma 5.3]{CG}).  

It follows that if $C$ is convex in $(G,c)$ and $g \in C$, then $\mathrm{rot}_c(g) = 0$, as powers of the lift $(g, -\eta(g)) \in \widetilde{G}_c$ are bounded in $<_c$.  Similarly $\mathrm{rot}_d(g) = 0$ for all $g \in C$.

On the other hand, suppose that $g \in G \setminus C$, no proper power of $g$ lies in $C$, and $j, k, \ell$ are pairwise distinct nonnegative integers.   Then the cosets $g^j C, g^kC, g^{\ell}C$ are pairwise distinct and so
\[c(g^j, g^k,g^{\ell}) = \overline{c}(g^jC, g^kC,g^{\ell}C) = \overline{d}(g^jC, g^kC,g^{\ell}C)=d(g^j, g^k,g^{\ell}).
\]
Thus, in the notation of Lemma \ref{open set}, $d \in U_n$ for all $n \geq 0$.  From Lemma \ref{open set} it therefore follows that for $\tilde{g} = (g,0) \in \widetilde{G}_c$ and $h = (g,0) \in \widetilde{G}_d$ we have $[h^n]_{d} = [\widetilde{g}^n]_c$ for all $n \geq 0$.  Therefore $\mathrm{rot}_c(g) = \mathrm{rot}_d(g)$.  

Now suppose that $g \in G \setminus C$ and some power of $g$ lies in $C$, say $m$ is the smallest positive integer such that $g^m \in C$.  Then for every $k \geq 1$,  $c(id, g^{km-1}, g^{km})$ and $c(id, g^{km}, g^{km+1})$ must have opposite signs, since $C$ is convex and both $g^{km-1}, g^{km+1} \notin C$.  Consequently $f_c(g^{km-1}, g) + f_c(g^{km}, g) = 1$, and likewise $f_d(g^{km-1}, g) + f_d(g^{km}, g) = 1$.  It follows from arguments similar to those appearing in Lemma \ref{open set} that the lifts $\widetilde{g} = (g,0) \in \widetilde{G}_c$ and $h = (h, 0) \in \widetilde{G}_d$ satisfy $[\widetilde{g}^n]_c = [h^n]_d$ whenever $n$ is not a multiple of $m$, and when $n$ is a multiple of $m$, these quantities differ by at most $1$.  In either case, it follows that $\mathrm{rot}_c(g) = \mathrm{rot}_d(g)$.

By arguing similarly, and using the same appropriately chosen lifts of $g,h \in G$ to compute $\tau_c(g,h)$ and $\tau_d(g,h)$, we can conclude that $\tau_c(g,h) = \tau_d(g,h)$ for all $g, h \in G$.

From Proposition \ref{semi classes} we conclude that $c$ and $d$ are semiconjugate.
\end{proof}

\begin{corollary}
\label{LOG_zero}
The map $\iota : \mathrm{LO}(G) \rightarrow \mathrm{CO}(G)$, which sends a left-ordering to the corresponding secret left-ordering, satisfies
\[ \iota(\mathrm{LO}(G)) = \{ c \in \mathrm{CO}(G) \mid \mathrm{rot}_c(g) = \tau_c(g,h) = 0 \mbox{ for all $g, h \in G$} \}.
\]
In particular, $\iota(\mathrm{LO}(G))$ is a semiconjugacy class.
\end{corollary}
\begin{proof}
Suppose that $c$ is a secret left-ordering.  One can check that $$P = \{g \in G \mid c(g^{-1}, id, g) =1 \}$$ is the positive cone of the unique left-ordering $<$ such that $c = c_<$.  Now define $\eta:G \rightarrow \{ 0, 1\}$ by  $\eta(id) =0$ and if $h \neq id$ then $\eta(h) =0$ if and only if $h \in P$.  Then as in the proof of Proposition \ref{prop convex}, the image of the inclusion $g \mapsto (g, -\eta(g))$ is a convex subgroup of $\widetilde{G}_{c}$ relative to the ordering $<_{c}$.  Thus for an arbitrary $g \in G$, powers of the lift $\widetilde{g} = (g, -\eta(g))$ are bounded in $\widetilde{G}_{c}$, and so we calculate $\widetilde{\mathrm{rot}}_c(\widetilde{g}) =0$.  As $g$ was arbitrary,  $\mathrm{rot}_c(g) = \tau_c(g,h) = 0$ for all $g,h \in G$.

Conversely, suppose that $c$ is a circular ordering satisfying $\mathrm{rot}_c(g) = \tau_c(g,h) = 0$ for all $g,h \in G$. Then $\widetilde{\mathrm{rot}}_c$ is an integer-valued homomorphism $\widetilde{\mathrm{rot}}_c : \widetilde{G}_{c} \rightarrow \mathbb{Z}$ satisfying $\widetilde{\mathrm{rot}}_c (id, 1) = 1$.  Define $\psi :G \rightarrow  \widetilde{G}_{c}$ by $\psi(g) = \widetilde{g}$, where $\widetilde{g}$ is the unique lift of $g$ satisfying $\widetilde{\mathrm{rot}}_c(\widetilde{g}) =0$ (in particular, this means that for all $g \in G$ either $\psi(g) = (g, 0)$ or $\psi(g) = (g, -1)$).  Then $\psi$ is an injective homomorphism, and so we can consider the left-ordering $<_c^{\psi} =(<_c)^{\psi} $ of $G$.

We claim that $c$ is the secret left-ordering associated to $<_c^{\psi}$.  To prove this, it suffices to show that $id <_c^{\psi} g_1 <_c^{\psi} g_2$ implies $c(id, g_1, g_2) = 1$.  To this end, note that $id <_c^{\psi} g_1 <_c^{\psi} g_2$ is equivalent to $id <_c \psi(g_1) <_c \psi(g_2)$, and as $\psi(g)$ is either $(g, 0)$ or $(g, -1)$ for all $g \in G$, this is equivalent to $id <_c (g_1, 0) <_c (g_2, 0)$.  Now the product $(g_1, 0)^{-1}(g_2, 0) = (g_1^{-1}g_2, f_c(g_1^{-1}, g_2)-1)$ must be positive in the ordering $<_c$ of $\widetilde{G}_c$, which requires $f_c(g_1^{-1}, g_2)=1$.  This implies that $c(id, g_1^{-1}g_2, g_1^{-1}) = c(id, g_1, g_2) =1$.
\end{proof}

\subsection{Semiconjugacy classes as subspaces}

Our main result in this section is that semiconjugacy classes are compact, $G$-invariant subspaces of $\mathrm{CO}(G)$.  Proofs of compactness of semiconjugacy classes in the context of representations $\rho:G \rightarrow \mathrm{Homeo}_+(S^1)$ usually involve continuity of $\mathrm{rot}:\mathrm{Homeo}_+(S^1) \rightarrow \mathbb{R} / \mathbb{Z}$ with respect to the $C^0$ topology.  Our approach is to use the fact that rotation number of a group element $g \in G$ depends continuously on the underlying circular ordering of $G$ (there is a remark to this effect in the proof of \cite[Lemma 2.13]{CW}).  

However, typical proofs of continuity of $\mathrm{rot}$ in the classical setting often involve approximating sequences of representations (e.g. see \cite[Lemma 4.6.2]{Cal1}). In our setting, this approach runs up against a slight hiccup---$\mathrm{CO}(G)$ may not be a sequential space, as the next example shows.

\begin{example}
The space $\mathrm{LO}(G)$ can be thought of as a compact subspace of $\mathrm{CO}(G)$, via the inclusion $\iota :\mathrm{LO}(G) \rightarrow \mathrm{CO}(G)$ sending $<$ to the secret left-ordering $c_<$ determined by $<$.  This example illustrates a group $G$ for which $\mathrm{CO}(G)$ is not sequentially compact, since the compact subspace $\mathrm{LO}(G) $ is not sequentially compact.

Set $I = \{ 0, 1\}^{\mathbb{N}}$ and let $G$ denote the direct sum of $I$ copies of $\mathbb{Z}$.  Denote the generator of the $i$-th copy of $\mathbb{Z}$ by $a_i$.  Let $\pi_n : I \rightarrow \{ 0, 1\}$ denote the projection map for each $n \in \mathbb{N}$.

It happens that any choice of signs for the generators $a_i \in G$ can be extended to a bi-ordering of $G$.\footnote{This can be proved by transfinite induction, using a lexicographic construction.}  Using this claim, we construct a sequence $\{P_n\}_{n=1}^{\infty}$ of positive cones as follows.  Let $P_n$ be a choice of positive cone that satisfies: $a_i \in P_n$ if $\pi_n(i) =1$, and $a_i^{-1} \in P_n$ otherwise.  Now let $\{P_{n_k}\}_{k=1}^{\infty}$ be a subsequence, and define $j \in I$ as follows.  Let $j$ be the sequence with $1$ appearing in position $n_{2k}$ for all $k$, and $0$ appearing in every other position.

Now consider the generator $a_j$.  By our choice of $j$, we have $\pi_{n_k}(j) = 1$ if $j$ is even, meaning $a_j \in P_{n_k}$ when $k$ is even.  On the other hand, $\pi_{n_k}(j) = 0$ if $j$ is odd, so 
$a_j^{-1} \in P_{n_k}$ when $k$ is odd.  Thus the subsequence $\{P_{n_k}\}_{k=1}^{\infty}$ cannot converge, since even terms lie in the open set $U_{a_j}$ and odd terms lie in the open set $U_{a_j^{-1}} = U_{a_j}^c$.

Despite these spaces not being sequential, there is no obstacle to proving continuity of the following maps if one mimics the existing proofs using convergent nets.

\label{not_sequential}
\end{example}

\begin{definition} Let $G$ be a circularly orderable group.
For each $g \in G$, define $\rho_g : \mathrm{CO}(G) \rightarrow \mathbb{R} /\mathbb{Z}$ by $\rho_g(c) = \mathrm{rot}_c(g)$.  For each pair $(g,h) \in G^2$ define $\tau_{g,h} : \mathrm{CO}(G) \rightarrow \mathbb{R}$ by $\tau_{g,h}(c) = \tau_c(g,h)$.
\end{definition}

\begin{proposition}
\label{conj continuous}
The map $\rho_g : \mathrm{CO}(G) \rightarrow \mathbb{R} /\mathbb{Z}$ is continuous for each $g \in G$.
\end{proposition}
\begin{proof}
Suppose that $\{ d_{\alpha} \}_{\alpha \in A}$ is a net converging to $c \in \mathrm{CO}(G)$.  For each $n \geq 0$, choose an open set $U_n$ containing $c$ as in Lemma \ref{open set}.  Since  $\{ d_{\alpha} \}_{\alpha \in A}$ converges to $c$, for each $n$ there exists an $\alpha_n \in A$ such that $\beta > \alpha_n$ implies $d_{\beta} \in U_n$.  Consequently if $h_\alpha = (g, 0) \in \widetilde{G}_{d_{\alpha}}$ for all $\alpha \in A$ and $\widetilde{g} = (g,0) \in \widetilde{G}_c$, then $[h_{\beta}^n]_{d_{\beta}} = [\widetilde{g}^n]_c$ whenever $\beta > \alpha_n$.   For $\beta > \alpha_n$ we have 
\[ |\tau(h_{\beta}) - \tau(\widetilde{g})| \leq \left| \tau(h_{\beta}) - \frac{[h_{\beta}^n]_{d_{\beta}}}{n} \right|+ \left|\frac{[\widetilde{g}^n]_c}{n} - \tau(\widetilde{g}) \right| \leq \frac{1}{n} + \frac{1}{n} = \frac{2}{n}
\]
where the second inequality above follows from the fact that $|\tau(h) - [h]_c/n|\leq \frac{1}{n}$ for every $c \in \mathrm{CO}(G)$ and $h \in \widetilde{G}_c$.  It follows that the net $\{ \rho_g(d_{\alpha}) \}_{\alpha \in A}$ converges to $\rho_g(c)$, so that $\rho_g$ is continuous.
\end{proof}

By a similar argument we can prove:

\begin{proposition}
\label{tau continuous}
The map $\tau_{g,h} : \mathrm{CO}(G) \rightarrow \mathbb{R}$ is continuous for each pair $(g,h) \in G^2$.
\end{proposition}

We conclude the following.

\begin{proposition}
\label{compact_invariant}
Semiconjugacy classes are compact subsets of $\mathrm{CO}(G)$, invariant under the $G$-action on $\mathrm{CO}(G)$ by conjugation.  
\end{proposition}
\begin{proof}
Define a continuous map $\rho: \mathrm{CO}(G) \rightarrow (\mathbb{R}/\mathbb{Z})^G$ by declaring $\rho(c)$ to be the element of $(\mathbb{R}/\mathbb{Z})^G$ that yields $\rho_g(c)$ upon projecting to the $g$-th factor.  Similarly define a continuous map $\tau : \mathrm{CO}(G) \rightarrow \mathbb{R}^{G^2}$ by declaring $\tau(c)$ to be the element of $\mathbb{R}^{G^2}$ that yields $\tau_{g,h}(c)$ upon projection to the $(g,h)$-th factor.

Sets of the form $\rho^{-1}(x)$ where $x \in (\mathbb{R}/\mathbb{Z})^G$ and $\tau^{-1}(y)$ where $y \in \mathbb{R}^{G^2}$ are closed, by continuity of the maps $\rho$ and $\tau$.  It follows that sets of the form $\rho^{-1}(x) \cap \tau^{-1}(y)$ are closed, hence compact.  By Theorem \ref{semi classes}, these sets correspond exactly to semiconjugacy classes, and by Proposition \ref{conj_inv}, these sets are invariant under the $G$-action.
\end{proof}

\section{Structure of semiconjugacy classes}
We begin with a basic lemma needed for the proof of Theorem \ref{Thm2}.

\begin{lemma}
\label{cts extension}
Suppose that
\[ 1 \rightarrow K \rightarrow G \stackrel{q}{\rightarrow} H \rightarrow 1
\] 
is a short exact sequence of groups, that $K$ is left-orderable and $H$ is circularly orderable.  For a fixed left-ordering $<$ of $K$, let $\phi_{<} : \mathrm{CO}(H) \rightarrow \mathrm{CO}(G)$ be the map defined by lexicographic extension of a circular ordering $c$ of $H$ by $<$, so that 
\[
\phi(c)(g_1,g_2,g_3) = \begin{cases}
c(q(g_1),q(g_2),q(g_3)) & \text{if } q(g_1),q(g_2),q(g_3) \text{ are all distinct,}\\
c_<(g_2^{-1}g_1,id,g_1^{-1}g_2) &\text{if } q(g_1) = q(g_2) \neq q(g_3),\\
c_<(id,g_1^{-1}g_2,g_1^{-1}g_3) &\text{if } q(g_1) = q(g_2) = q(g_3).
\end{cases}
\]
Then $\phi_<$ is continuous.
\end{lemma}
\begin{proof}
Consider a basic open set $U_t^i$ in $\mathrm{CO}(G)$, where $t = (g_1, g_2, g_3)$ is a triple of distinct elements of $G$.  We consider three cases.
  
\noindent \textbf{Case 1.} If $q(g_1), q(g_2), q(g_3)$ are all distinct, then $\phi_<^{-1}(U_t^i) = U_{q(t)}^i$, where $q(t) = (q(g_1), q(g_2), q(g_3))$.

\noindent \textbf{Case 2.} If $q(g_1) = q(g_2) \neq q(g_3)$, then $\phi_<^{-1}(U_t^i) = \emptyset$ if $c_<(g_2^{-1}g_1,id,g_1^{-1}g_2) = i$, and $\phi_<^{-1}(U_t^i) = \mathrm{CO}(H)$ otherwise.

\noindent \textbf{Case 3.} If $q(g_1) = q(g_2) = q(g_3)$, then $\phi_<^{-1}(U_t^i) = \emptyset$ if $c_<(id,g_1^{-1}g_2,g_1^{-1}g_3) = i$, and $\phi_<^{-1}(U_t^i) = \mathrm{CO}(H)$ otherwise.

In any event, the preimage of every subbasic open set is open, so $\phi_<$ is continuous.
\end{proof}

\label{main_proof}

\begin{proof}[Proof of Theorem \ref{Thm2}]
Let $G$ be a circularly orderable group and $S \subset \mathrm{CO}(G)$ a semiconjugacy class.
 Let $M$ be a minimal invariant set of the $G$-action on $S$, such a set exists since $S$ is both compact and $G$-invariant by Proposition \ref{compact_invariant}. If $M$ is infinite, then $M$ has no isolated points since every orbit in $M$ is dense, so $M$ is homeomorphic to the Cantor set and is uncountable.

On the other hand, suppose that $M$ is finite.  Then the stabilizer of any $c \in M$ under the $G$-action is a finite-index subgroup of $G$, and so by \cite[Corollary 2.13]{CMR}, the linear part $H$ of $(G,c)$ is normal and the natural ordering inherited by $G/H$ is Archimedean.  In particular, $G/H$ is order-isomorphic to a subgroup $C$ of $S^1$ with its usual ordering, with $\mathrm{rot}_c: G/H \rightarrow C$ serving as the isomorphism \cite[Corollary 2.12]{CMR}.  Thus $G$ sits in a short exact sequence
\[ 1 \rightarrow H \rightarrow G \stackrel{\mathrm{rot}_c}{\longrightarrow} C \rightarrow 1,
\]
and by Proposition \ref{prop convex}, every circular ordering of $G$ arising lexicographically from this short exact sequence lies in $S$.  

Now if $H$ admits infinitely many left-orderings, then it admits uncountably many left-orderings \cite{Li}, hence there are infinitely many circular orderings that arise lexicographically from the short exact sequence above.  Thus $S$ is uncountable in this case.  Otherwise $H$ admits only finitely many left-orderings, so $H$ is a Tararin group $T_k$ admitting exactly $2^k$ left-orderings for some $k \geq 1$.  Then by fixing the circular ordering of $C$ and changing the left-ordering of $T_k$ we can lexicographically construct exactly $2^k$ circular orderings on $G$ that lie in $S$.  Let $S'$ denote the set of all orderings constructed in this way.  

On the other hand, suppose that $d \in S$ is a circular ordering of $G$.  Then given $c \in S'$, since $d$ is semiconjugate to $c$ we have $\mathrm{rot}_c = \mathrm{rot}_d$ by Proposition \ref{semi classes}.  Moreover $\mathrm{rot}_d$  is order-preserving.  Thus $d$ arises lexicographically from the same short exact sequence as $c$, meaning $d \in S'$.  We conclude $S = S'$ and so $S$ is either uncountable or $|S| = |\mathrm{LO}(T_k)| = 2^k$.  

Next, suppose that $\mathrm{CO}(G)$ is not finite, so that either $C$ is infinite, or $C$ is finite and the action of $C$ on $T_k/T_{k-1}$ is trivial.  We must show that every ordering in $S$ is not isolated.
First, if $C$ is infinite, then by \cite[Theorem B]{CMR} the space of circular orderings of $C$ is homeomorphic to a Cantor set.  It follows from Lemma \ref{cts extension} that every ordering of $G$ that lies in $S$ is not isolated, since it arises from the short exact sequence \[ 1 \rightarrow T_k \rightarrow G \stackrel{\mathrm{rot}}{\longrightarrow} C \rightarrow 1.
\]

On the other hand, if $C$ is finite (hence cyclic), and if the action of $C$ on $T_k/T_{k-1}$ is trivial, then $G/T_{k-1}$ is an infinite abelian group.  Then the orderings of $S$ can be approximated by changing the circular ordering on $G/T_{k-1}$ and again applying Lemma \ref{cts extension},  since $\mathrm{CO}(G/T_{k-1})$ is homeomorphic to a Cantor set.  

\end{proof}

There is an immediate corollary of this result which may be of some independent interest.

\begin{corollary}
Every circular ordering of a group $G$ is semiconjugate to a non-isolated circular ordering, unless $\mathrm{CO}(G)$ is finite.
\end{corollary}

\section{Left-orderings as circular orderings}

\label{genuine}

In this section we will study the subspace of secret left-orderings, and the closure of its complement. Following Baik and Samperton \cite{BS}, define $\mathrm{CO}_g(G)$ to be $\mathrm{CO}(G)\setminus \iota(\mathrm{LO}(G))$. The elements of $\mathrm{CO}_g(G)$ are called \textit{genuine} circular orderings of $G$, by Corollary \ref{LOG_zero} they are precisely the circular orderings $c$ for which either $\mathrm{rot}_c$ or $\tau_c$ is not the zero function.

We begin by recalling a standard construction.

\begin{construction}
\label{quotient}
Suppose we are given a group $G$ with a left-ordering $<$, and that $z \in G$ is central and cofinal with respect to $<$.   Define a circular ordering $c$ on $G/\langle z \rangle$ as follows:  for every $g\langle z \rangle \in G/\langle z \rangle$, define the {\em minimal representative} of $g\langle z \rangle$ to be the unique $\overline{g} \in g\langle z \rangle$ satisfying that $id \leq \overline{g} <z$. Then set 
$$c(g_1\langle z \rangle, g_2\langle z \rangle,g_3\langle z \rangle)=sign(\sigma),$$
where $\sigma$ is the unique permutation in $S_3$ such that $\overline{g_{\sigma(1)}}<\overline{g_{\sigma(2)}}<\overline{g_{\sigma(3)}}$.  
\end{construction}

\begin{theorem}
\label{approx theorem}
Suppose that $(H, <)$ is a left-ordered group with positive cofinal central element $z \in H$. Suppose further that $\phi: L \rightarrow H$ is an embedding of a group $L$ into $H$ such that $\phi(L)$ is cofinal.  Then the secret left-ordering $c_{<^{\phi}}$ of $L$ is an accumulation point of $\mathrm{CO}_g(L)$.
\end{theorem}
\begin{proof}
We first construct a sequence of genuine circular orderings in $\mathrm{CO}_g(H)$ converging to $c_<$.  To do this, we first let $c_n$ be the circular order on $H/\langle z^n\rangle$ defined as in Construction \ref{quotient}.  That is, for any $h\in H$, let $[h]_n$ be the unique coset representative of $h\langle z^n\rangle$ with $id\leq   [h]_n <z^n$, define  $c_n(g_1\langle z^n\rangle, g_2\langle z^n\rangle, g_3\langle z^n\rangle)={\rm sign}(\s)$ where $\s$ is the unique permutation such that $[g_{\s(1)}]_n<[g_{\s(2)}]_n<[g_{\s(3)}]_n$.\footnote{For this proof alone, the notation $[ \cdot]_n$ will mean a particular choice of coset representative whenever the subscript is an integer.}  Now for every positive integer $n$, define a circular order $d_n$ on $H$ lexicographically by equipping $H/\langle z^n\rangle$ with $c_n$, the subgroup $\langle z^n \rangle$ with the unique ordering in which $z^n >1$, and using the short exact sequence  $$1\longrightarrow \langle z^n\rangle\longrightarrow H\stackrel{q_n}{\longrightarrow}{H}/{\langle z^n\rangle}\longrightarrow 1.$$

We check that the sequence $\{ d_n \}_{n=1}^{\infty}$ converges in $\mathrm{CO}(H)$ to $c_<$ as follows.  Consider a subbasic open set $$U_t^{c_<(t)}=\{d \in {\rm CO}(H) \mid  d(t)=c_<(t)\}$$
where $t=(g_1, g_2, g_3)\in H^3$ with all $g_i$ distinct.   We will deal only with the case of $c_<(t) = + 1$, and we fix a permutation $\s$ satisfying $g_{\s(1)} <g_{\s(2)}<g_{\s(3)}$ whose sign is therefore $+1$, the other cases are similar.  First observe that by choosing $N_1$ sufficiently large, one may assume that $q_n(g_i)$ are all distinct for $n>N_1$, so that $d_n(g_1, g_2, g_3)=c_n(q_n(g_1), q_n(g_2), q_n(g_3))$.  We now consider cases.

\noindent \textbf{Case 1.} Suppose either $g_i >id$ for $i=1,2,3$ or $g_i < id$ for $i=1,2,3$.  We will address the case of $id < g_i$ for $i=1,2,3$, the other case being almost identical.  In this case, choose $N_2$ such that $z^{N_2} > \max\{g_1, g_2, g_3\}$.  Then $[g_i]_n = g_i$ for $i = 1, 2, 3$ whenever $n>N_2$ and therefore 
\[ c_n(q_n(g_1), q_n(g_2), q_n(g_3)) = \mathrm{sign}(\s). 
\]
Thus $n \geq \max\{N_1, N_2\}$ implies $d_n \in U_t^{1}$.

\noindent \textbf{Case 2.} Suppose $g_{\s(1)} < id < g_{\s(2)} < g_{\s(3)}$.  Choose $N_3$ such that $g_{\s(1)}z^{N_3} > g_{\s(3)}$.  Then for $n > N_3$, $[g_{\s(1)}]_n = g_{\s(1)}z^n$, $[g_{\s(2)}]_n = g_{\s(2)}$ and $[g_{\s(3)}]_n = g_{\s(3)}$.  Then 
\[c_n(q_n(g_1), q_n(g_2), q_n(g_3)) = \mathrm{sign}(\tau)
\]
where $\tau$ is the permutation $\tau(1) = \s(2), \tau(2) = \s(3), \tau(3) = \s(1)$.  Thus $\tau = \s \circ \beta$, where $\beta = (1 \hspace{1em} 2 \hspace{1em} 3)$.  Therefore $\mathrm{sign}(\tau) = \mathrm{sign}(\s)$ =1.  We conclude that $n \geq \max\{N_1, N_3\}$ implies $d_n \in U_t^{1}$.

\noindent \textbf{Case 3.} Suppose $g_{\s(1)}  < g_{\s(2)}< id < g_{\s(3)}$.  Choose $N_4$ such that $g_{\s(1)}z^{N_4} > g_{\s(3)}$, then for $n>N_4$ we have $[g_{\s(1)}]_n = g_{\s(1)}z^n$, $[g_{\s(2)}]_n = g_{\s(2)}z^n$ and $[g_{\s(3)}]_n = g_{\s(3)}$.  We conclude as before that 
\[c_n(q_n(g_1), q_n(g_2), q_n(g_3)) = \mathrm{sign}(\tau)
\]
where $\tau(1) = \s(3), \tau(2) = \s(1), \tau(3) = \s(2)$, so $\mathrm{sign}(\tau) = \mathrm{sign}(\s)=1$.  We conclude as in the previous case that
whenever $n \geq \max\{N_1, N_4\}$,  $d_n \in U_t^{1}$.

Thus $\{d_n\}_{n=1}^{\infty}$ converges to $c_<$ in $\mathrm{CO}(H)$.  Now note that there is a continuous map $\psi: \mathrm{CO}(H) \rightarrow \mathrm{CO}(L)$ given by $\psi(c) = \phi^*c$, where $\phi^*c$ is the circular ordering of $L$ defined by
\[ \phi^*c(h_1, h_2, h_3) = c(\phi(h_1), \phi(h_2), \phi(h_3)).
\]  The map $\psi$ also satisfies $\psi(c_<) = c_{<^{\phi}}$.  Thus, to prove the proposition it suffices to show that the set $\{\psi(d_n)\}_{n=1}^{\infty}$ is infinite.  To do this, it suffices to check that the restriction orderings $r(d_n) = d_n|_{\phi(L)}$ provide infinitely many distinct circular orderings of $\phi(L)$, since the map $d_n|_{\phi(L)} \mapsto \phi^*d_n$ is injective.  We prove this claim below.

To this end, choose $g \in \phi(L)$ with $z<g$, this is possible since $\phi(L)$ is cofinal.  We will establish the claim by showing that the set $\{ \mathrm{rot}_{d_n}(g) \}_{n=1}^{\infty}$ is infinite.  Given $k, n >0$ use $a_{k,n}$ to denote the unique integer such that $(z^n)^{a_{k,n}} \leq g^k < (z^n)^{a_{k,n}+1}$.  
From our constructions, $\mathrm{rot}_{d_n}(g) = \mathrm{rot}_{c_n}(g\langle z^n \rangle)$, so we need to consider a lift of $g\langle z^n \rangle$ in $ \widetilde{H/\langle z^n \rangle}_{c_n}$.  Thankfully there is an order isomorphism 
\[ \Phi_n :  (H, <) \rightarrow (\widetilde{H/\langle z^n \rangle}_{c_n}, <_{c_n})
\]
which we can describe as follows.  We first observe that every element of $(H, <)$ can be written uniquely as $[g]_n (z^n)^{\ell}$ where $\ell \in \mathbb{Z}$; the isomorphism is then $\Phi_n ([g]_n (z^n)^{\ell}) = (g\langle z^n \rangle, \ell)$  \cite[Proposition 2.9]{CG}.  Choose $M$ such that $z^M >g$, then when $n \geq M$ we have $[g]_n = g$, so for arbitrary $k>0$ the isomorphism $\Phi_n$ yields
\[ (id, \ell) \leq_{c_n} (g\langle z^n\rangle, 0)^k <_{c_n} (id, \ell +1) \mbox{ if and only if } (z^n)^{\ell} \leq g^k < (z^n)^{\ell+1}.
\]
We conclude that $[(g\langle z^n \rangle, 0)^k]_{c_n} = a_{k,n}$ for all $k>0$ and $n>M$.
Thus we may choose the lift $(g\langle z^n \rangle, 0)$ of $g\langle z^n \rangle$ to compute $\mathrm{rot}_{c_n}(g\langle z^n \rangle)$, and begin by computing
  \[ \widetilde{\mathrm{rot}}_{c_n}((g\langle z^n \rangle, 0)) = \lim_{k \to \infty}\frac{[(g\langle z^n \rangle, 0)^k]_{c_n}}{k} = \lim_{k \to \infty}\frac{a_{k,n}}{k}.
  \]

From our definitions, observe that $a_{k,n} = \lfloor \frac{a_{k,1}}{n} \rfloor$.  Further since $z<g$ we have $z^n < g^n$, so $a_{k,1}>k$ and $\lim \frac{a_{k,1}}{k} = s  \geq 1$.  Thus
\[ \widetilde{\mathrm{rot}}_{c_n}((g\langle z^n \rangle, 0)) = \lim_{k \to \infty}\frac{\lfloor \frac{a_{k,1}}{n} \rfloor}{k} =\frac{s}{n} > 0,
\]
and so $\mathrm{rot}_{d_n}(g) =  \frac{s}{n} \mod \mathbb{Z}$.   Choose $N$ such that $\frac{s}{N}<1$, so that for $n >N$ we have $0<\frac{s}{n}<1$ and therefore $ \frac{s}{n} \mod \mathbb{Z} = \frac{s}{n}$.  It follows that whenever $n, m >N$, $\mathrm{rot}_{d_n}(g)$ and $\mathrm{rot}_{d_m}(g)$ are distinct.  This proves the claim.
\end{proof}

\begin{proposition}
\label{genuine closure}
Suppose that $G$ is a group, and that $(H, <)$ is a left-ordered group admitting a cofinal central element and that $\phi : G \rightarrow H$ is a homomorphism whose image is cofinal.

If $\prec$ is a lexicographic left-ordering of $G$ that arises from the short exact sequence 
\[ 1 \rightarrow \ker \phi \rightarrow G \rightarrow \phi(G) \rightarrow 1
\]
where $\phi(G)$ is equipped with the restriction of $<$, then $c_{\prec}$ is an accumulation point of $\mathrm{CO}_g(G)$.
\end{proposition}
\begin{proof}
Let $d : \phi(G)^3 \rightarrow \{\pm 1, 0\}$ denote the secret left-ordering of $\phi(G)$ arising from restriction of $<$ to $\phi(G)$.  By Theorem \ref{approx theorem}, we can choose a sequence of genuine circular orderings $\{d_n\}_{n=1}^{\infty}$ in $\mathrm{CO}_g(\phi(G))$ such that $\lim_{n \to \infty} d_n = d$.

Lexicographically create a sequence of circular orderings $\{c_n \}_{n=1}^{\infty}$ of $G$ by using the restriction of $\prec$ to left-order $\ker \phi$, and using $d_n$ to order the quotient $\phi(G)$.  
The circular orderings $c_n$ are all genuine circular orderings of $G$, because $\mathrm{rot}_{c_n}(g) = \mathrm{rot}_{d_n}(\phi(g))$ for all $g \in G$.  By Lemma \ref{cts extension}, the sequence $\{c_n \}_{n=1}^{\infty}$ converges to $c_{\prec}$.
\end{proof}

\begin{corollary}
For the following groups, $\overline{\mathrm{CO}_g(G)} = \mathrm{CO}(G)$.  In particular, whenever such a group is left-orderable, every left-ordering is an accumulation point of genuine circular orderings.
\begin{enumerate}
\item Fundamental groups of closed Seifert fibred manifolds with base orbifold $S^2(\alpha_1, \dots, \alpha_n)$.
\item Braid groups.
\item Finitely generated nilpotent groups.
\end{enumerate}
\end{corollary}
\begin{proof}
In each of the cases, if the group in question is not left-orderable then $\overline{\mathrm{CO}_g(G)} = \mathrm{CO}(G)$ is immediate.  We therefore only consider the left-orderable cases below.

Let $M$ be a Seifert fibred space as above, then 
\[ \pi_1(M) = \langle \gamma_1, \dots, \gamma_n, h \mid \mbox{$h$ central}, \gamma_i^{\alpha_i} = h^{\beta_i}, \gamma_1 \dots \gamma_n = h^b \rangle
\]
where $(\alpha_i, \beta_i)$ are pairs of relatively prime integers and $b \in \mathbb{Z}$.
Let $<$ be an arbitrary left-ordering of $\pi_1(M)$ and observe that $$C = \{ g \in \pi_1(M) \mid \exists k\in \mathbb{Z} \mbox{ such that } h^{-k} < g < h^k \}$$ is a subgroup of $\pi_1(M)$, and that $h$ is cofinal in $C$.   The relations $ \gamma_i^{\alpha_i} = h^{\beta_i}$ guarantee that $\gamma_i \in C$ for all $i$, so that $C = \pi_1(M)$ and the conclusion follows from Proposition \ref{genuine closure}.

Recall that the braid group $B_n$ has generators $\sigma_1, \ldots , \sigma_{n-1}$, and relations $\sigma_i \sigma_j = \sigma_j \sigma_i$ if $|i-j|>1$, and $\sigma_i \sigma_j \sigma_i = \sigma_j \sigma_i \sigma_j$ if $|i-j|=1$. The square of the Garside half-twist $\Delta_n^2$ is the generator of the centre of $B_n$.  As before, consider the subgroup $$C = \{ \beta \in B_n \mid \exists k\in \mathbb{Z} \mbox{ such that } (\Delta_n^2)^{-k} < \beta < (\Delta_n^2)^k \}.$$  The braids $\delta_n = \s_1 \s_2 \cdots \s_{n-1}$, and $\varepsilon_n = \s_1^2 \s_2 \cdots \s_{n-1}$, satisfy $\delta_n^n = \varepsilon_n^{n-1} = \Delta_n^2$, and are therefore both contained in $C$.  Therefore so is their product $\varepsilon_n\delta_n^{-1} = \sigma_1$.  For any $i=1, \cdots, n-1$, there exists a $k$ such that $\delta^k \s_1 \delta^{-k} = \s_i$, and thus $\s_i \in C$.  We conclude $B_n=C$ and apply Proposition \ref{genuine closure}.

If $G$ is a finitely generated torsion-free nilpotent group, then $G$ is left-orderable and every left-ordering is Conradian (see \cite{Au}).  It follows that every left-ordering is lexicographic relative to a short exact sequence
\[ 1 \rightarrow K \rightarrow G \rightarrow \mathbb{Z}^k \rightarrow 1
\]
for some $k \geq 1$, so Proposition \ref{genuine closure} applies to every left-ordering of such a group.
\end{proof}

At the other extreme, one might wonder if it is possible that $\iota(\mathrm{LO}(G)) \cap \overline{\mathrm{CO}_g(G)} = \emptyset$---that is, whether or not $\iota(\mathrm{LO}(G))$ can ever be an open subset of $\mathrm{CO}(G)$.  This can only happen, if ever, for a very special kind of group.    Recall that a subgroup $G$ of $\mathrm{Homeo}_+(\mathbb{R})$ is \emph{locally contracting} if for every $x \in \mathbb{R}$ there exists $y>x$ and a sequence of elements $\{g_n\}_{n=1}^{\infty}$ of $G$ such that $\lim_{n \to \infty} g_n(x) = \lim_{n \to \infty} g_n(y)$ exists.  A subgroup is called globally contracting if such a sequence of group elements exists for every interval $[x,y] \subset \mathbb{R}$. Following \cite[Proposition 3.5.20]{DNR} and \cite[Theorem 1]{Mal}, we can define three mutually exclusive types of left-orderings on a countable group $G$:
\begin{enumerate}
\item \textbf{Type I.} Left-orderings whose dynamical realization is semiconjugate to an action by translations.
\item \textbf{Type II.} Left-orderings whose dynamical realization is semiconjugate to some $\rho :G \rightarrow \widetilde{\mathrm{Homeo}_+}(S^1)$ whose image is a minimal, locally contracting subgroup.
\item \textbf{Type III.} Left-orderings whose dynamical realization is semiconjugate to some $\rho :G \rightarrow \widetilde{\mathrm{Homeo}_+}(S^1)$ whose image is globally contracting.
\end{enumerate}

In \cite{HLNR}, the authors call a finitely-generated group all of whose actions on $\mathbb{R}$ are type III a ``left-orderable monster", where they also prove the existence of such groups.

\begin{corollary}
If $G$ is a finitely generated left-orderable group such that $\iota(\mathrm{LO}(G)) \cap \overline{\mathrm{CO}_g(G)} = \emptyset$, then $G$ is a left-orderable monster.
\end{corollary}
\begin{proof}
If $G$ admits a left-ordering of either type I or type II, then $G$ admits either a homomorphism onto a finitely-generated torsion-free abelian group, or a homomorphism into $\widetilde{\mathrm{Homeo}_+}(S^1)$.  In either case, the resulting left-ordering is an accumulation point of genuine circular orderings of $G$, by Proposition \ref{genuine closure}.
\end{proof}

\begin{question}
Does there exist a left-orderable monster admitting genuine circular orderings? If yes, is it the case that $\iota(\mathrm{LO}(G)) \cap \overline{\mathrm{CO}_g(G)} = \emptyset$? 
\end{question}

\end{document}